\documentclass[11pt]{article}

\usepackage{amsthm}
\usepackage{amssymb}
\usepackage{amsmath,enumerate}
\usepackage[ruled, lined, linesnumbered]{algorithm2e}
\usepackage{comment}
\usepackage{thm-restate}
\usepackage[T1]{fontenc}
\usepackage{lmodern}
\usepackage{url} 
\usepackage{hyperref}
\usepackage{graphicx}
\usepackage{subfigure}
\usepackage[noabbrev,capitalise]{cleveref}
\usepackage[affil-it]{authblk}
\usepackage{color}

\usepackage[margin=1in]{geometry}

\theoremstyle{plain}
\newtheorem{thm}{Theorem}[section]
\newtheorem{lem}[thm]{Lemma}

\newtheorem{cor}[thm]{Corollary}

\newtheorem{conj}[thm]{Conjecture}

{\noindent \emph{Proof.} {}{#1}{}}{\hfill
	$\Diamond$\vspace{1em}}

\theoremstyle{plain} 
\newcommand{\thistheoremname}{}
\newtheorem{genericthm}[section]{\thistheoremname}

\theoremstyle{definition}

\def\es{\emptyset}
\def\less{\setminus}
 \def\dfn#1{{\sl #1}}

\newcommand{\etalchar}[1]{$^{#1}$}
\DeclareTextAccent{\caronaccent}{T1}{7}

\newcommand{\Nesetril}
  {Ne{\caronaccent{s}}et{\caronaccent{r}}il}

\title{On the size of  $(K_{t_1}, \ldots, K_{t_k})$-co-critical graphs}
\author{Zi-Xia Song\thanks{Supported by  NSF award DMS-2153945. E-mail address: {\tt Zixia.Song@ucf.edu}.}}
  \affil{ 
  { \small {Department  of Mathematics, University of Central Florida, Orlando, FL 32816, USA}}  
     }

\date{}
\begin{document}
\maketitle
\begin{abstract}
 Given integers $k\ge2$ and $t_1, \ldots, t_k\ge2$,   we write \emph{$G \rightarrow (K_{t_1}, \ldots, K_{t_k})$} if  every $k$-coloring of the edges of a graph $G$ contains  a monochromatic copy of $K_{t_i}$ in color $i$ for some $i\in\{1, \ldots, k\}$. A non-complete graph $G$ is \emph{$(K_{t_1}, \ldots, K_{t_k})$-co-critical} if $G  \nrightarrow (K_{t_1}, \ldots, K_{t_k})$,   but  $G+e\rightarrow (K_{t_1}, \ldots, K_{t_k})$ for every edge $e\notin E(G)$.  Let $r=R(K_{t_1}, \ldots, K_{t_k})$ denote the Ramsey number.   In 1987, Hanson and Toft conjectured   that   every  $(K_{t_1}, \ldots, K_{t_k})$-co-critical graph $G$ on $n\ge r$ vertices satisfies   \[|E(G)|\ge (r-2)n-   \binom{r- 1}{2}.\] 
 This bound is best possible for every  $n\ge r$. More recently, the present author conjectured that  every such graph has minimum degree at least $r-2$.  Using  the $q$-neighbor bootstrap percolation closure method,  here we prove that the Hanson-Toft Conjecture holds asymptotically, provided that the minimum-degree conjecture is true;  more precisely, If every
$(K_{t_1},\ldots,K_{t_k})$-co-critical graph has minimum degree at least
$r-2$, then there is a constant $C=C(r,k)$
such that every $(K_{t_1},\ldots,K_{t_k})$-co-critical graph $G$ on
$n\ge r$ vertices satisfies  $|E(G)|\ge (r-2)n-C$.  
 \end{abstract}

\baselineskip 18pt

\section{Introduction}
   All graphs considered in this paper are finite, and without loops or multiple edges. For a graph $G$, we will use $V(G)$ to denote the vertex set, $E(G)$ the edge set,   $e(G)$ the number of edges,   $\delta(G)$ the minimum degree,   and $\overline{G}$ the complement of $G$.
 For  any  edge  $e$ in $ \overline{G} $, we use $G+e$ to denote the graph obtained from $G$ by adding the new edge $e$. 
The {\dfn{join}} $G+H$   of two 
vertex disjoint graphs
$G$ and $H$ is the graph having vertex set $V(G)\cup V(H)$  and edge set $E(G)
\cup E(H)\cup \{xy\mid x\in V(G),  y\in V(H)\}$.
  For any positive integer $k$, we write  $[k]$ for the set $\{1,2, \ldots, k\}$. We use the convention   ``$A:=$'' to mean that $A$ is defined to be the right-hand side of the relation.  \medskip

  A \dfn{$k$-edge coloring} of a graph $G$ is a function $\tau:E(G)\to [k]$.  We think of the set  $[k]$ as a set of colors, and we may identify a member of $[k]$ as a color, say, color $k$ is blue. Given an integer $k \ge 2$ and graphs $G$, ${H}_1, \ldots, {H}_k$,   we write \dfn{$G \rightarrow ({H}_1, \dots, {H}_k)$} if every $k$-coloring of $E(G)$ contains a monochromatic  ${H}_i$ in color $i$ for some $i\in [k]$.
  The classical \dfn{Ramsey number}  $R({H}_1, \dots, {H}_k)$ is the minimum positive integer $n$ such that $K_n \rightarrow ({H}_1, \dots, {H}_k)$. \medskip

  Following~\cite{Galluccio1992}, a non-complete graph $G$ is $(H_1, \dots, H_k)$-\dfn{co-critical} if $G \not\rightarrow ({H}_1,\dots, {H}_k)$, but $G + e \rightarrow ({H}_1,\dots, {H}_k)$ for every edge $e$ in $\overline{G}$. The notion of co-critical graphs was initiated by   \Nesetril~\cite{Nesetril1986}  in 1986   when he asked the following question regarding   $(K_3, K_3)$-co-critical graphs: 
\begin{quote}
 Are there   infinitely  many \dfn{minimal} co-critical graphs, i.e.,  co-critical graphs which lose this property when any vertex is deleted? Is $K_6^-$ the only one? 
\end{quote}
This was answered in the positive by Galluccio, Simonovits and Simonyi~\cite{Galluccio1992}.  They constructed infinitely many minimal $(K_3, K_3)$-co-critical graphs without   $K_5$ as a subgraph.    Szab\'o~\cite{Szabo1996} then constructed   infinitely many  nearly regular $(K_3, K_3)$-co-critical graphs with low maximum degree.  It remains unknown whether there are infinitely many \dfn{strongly} minimal co-critical graphs, where an $(H_1, \ldots,  H_r)$-co-critical graph is \dfn{strongly minimal co-critical} if it contains no proper subgraph which is also $(H_1, \ldots,  H_r)$-co-critical.  Galluccio, Simonovits and Simonyi~\cite{Galluccio1992}  also made some observation on the  maximum number of possible edges of $(H_1, \ldots,  H_r)$-co-critical graphs. \medskip

    Hanson and Toft~\cite{Hanson1987}  independently   studied  the minimum and maximum number  of edges over all $(H_1, \ldots,  H_k)$-co-critical graphs on $n$ vertices when $H_1, \ldots,  H_k$ are complete graphs, under the name of \dfn{strongly $(|H_1|, \ldots, |H_k|)$-saturated} graphs. Recently, this topic has  been studied under the name of \dfn{$\mathcal{R}_{\min}(H_1, \ldots, H_k)$-saturated} graphs \cite{Chen2011, Ferrara2014,  RolekSong18b}.   Let $k\ge2$ and  $t_1,\ldots, t_k\ge2$ be integers. Let $ r:=R(K_{t_1},\dots,K_{t_k})$.  Hanson and Toft~\cite{Hanson1987}  observed   that for all $n \ge r$,  the graph $K_{r-2}+ \overline K_{n-r+2}$ is $(K_{t_1},\dots, K_{t_k})$-co-critical with $(r - 2)n  - \binom{r - 1}{2}$ edges; they made the following  conjecture.

\begin{conj}[Hanson and Toft~\cite{Hanson1987}]\label{HTC}  
Let  $r = R(K_{t_1}, \dots, K_{t_k})$. If $G$ is a $(K_{t_1},\dots, K_{t_k})$-co-critical graph on $n\ge r$ vertices, then 
\begin{align*}
e(G)\ge (r - 2)n - \binom{r - 1}{2}.  
\end{align*}
   The bound is best possible for all $n\ge r$.
\end{conj}
\medskip

 It was shown in~\cite{Chen2011} that every $(K_3,K_3)$-co-critical graph on $n\ge 56$ vertices has at least $4n-10$ edges. This settles the first non-trivial case of Conjecture~\ref{HTC} for sufficiently large $n$.    Conjecture~\ref{HTC} remains wide open. 
 We refer the reader to a recent paper by Zhang and   the present  author~\cite{SZ21} for further background on $(H_1, \dots, H_k)$-co-critical graphs, and to~\cite{P3cocritical,CMSZ22,C4cocritical,DSY22, Ferrara2014, RolekSong18b,SZ21}   for recent work on minimizing the number of edges in  $(H_1, \ldots, H_k)$-co-critical graphs. We also refer the reader to    the dynamic survey~\cite{CFFS21} on the extensive studies on $K_t$-saturated graphs. \medskip
 
 In the same paper  Galluccio, Simonovits and Simonyi  also studied   the minimum  degree of $(K_3, K_3)$-co-critical graphs.

\begin{thm}[Galluccio,  Simonovits and Simonyi~\cite{Galluccio1992}]\label{t:K3K3}
Every $(K_3, K_3)$-co-critical graph on $n\ge6$ vertices  has minimum degree at least four. The bound is sharp for all $n\ge 6$.   

\end{thm}
 It is worth noting that  the graph $K_{r-2}+ \overline K_{n-r+2}$ has minimum degree $r-2$. With the support of \cref{t:K3K3}, the present author~\cite{CSS24}    made the following conjecture.

\begin{conj}[Song~\cite{CSS24} ]\label{c:mindeg}
Let $r=R(K_{t_1},\dots,K_{t_k})$ and let $G$ be a $(K_{t_1}, \ldots,  K_{t_k})$-co-critical graph on $n\ge r$ vertices.  Then $\delta(G)\ge r-2$. 
\end{conj}
 
Very recently, Casas-Rocha, Snyder and the present author~\cite{CSS24} confirmed \cref{c:mindeg} for $(K_3, K_4)$-co-critical graphs.  

\begin{thm}[Casas-Rocha, Snyder and Song~\cite{CSS24}]\label{t:K3K4}
Every $(K_3, K_4)$-co-critical graph on $n\ge 9$ vertices has  minimum degree at least seven. The bound is sharp for all $n\ge 9$.
\end{thm}

The purpose  of this paper is  to prove that Hanson-Toft Conjecture holds asymptotically under the assumption that Conjecture~\ref{c:mindeg} is true.

\begin{restatable}{thm}{threeclaw}\label{t:main}
Let
$r=R(K_{t_1},\ldots,K_{t_k})$, where  $k\ge2$ and $t_1,\ldots,t_k\ge2$ are  integers.  If every
$(K_{t_1},\ldots,K_{t_k})$-co-critical graph has minimum degree at least
$r-2$, then there is a constant $C=C(r,k)$
such that every $(K_{t_1},\ldots,K_{t_k})$-co-critical graph $G$ on
$n\ge r$ vertices satisfies
\[
  e(G)\ge (r-2)n-C.
\]
\end{restatable}

Using the  $q$-neighbor bootstrap percolation closure method, we  prove \cref{t:main} in Section~\ref{s:main}. The $q$-neighbor bootstrap percolation process was introduced by Chalupa, Leath and Reich~\cite{bootstrap}. Day~\cite{Day2017} applied the $q$-neighbor bootstrap percolation closure method to $K_t$-saturated graphs with prescribed minimum degree to determine the minimum number of edges. In particular, for $t=3$, he confirmed a conjecture of Bollob\'as~\cite{Bollobas}. Building on Day's work, Zhang and the present author~\cite{SZ21} extended the method to graphs with prescribed minimum degree that are not necessarily $K_t$-saturated and used it to investigate the minimum number of edges in $(K_t,\mathcal{T}_k)$-co-critical graphs. Building on  the  work in~\cite{SZ21},  we prove \cref{t:main}  by applying  the $q$-neighbor bootstrap percolation closure method directly to $(K_{t_1},\ldots,K_{t_k})$-co-critical graphs with prescribed minimum degree. \medskip

 Combining  \cref{t:main}  with \cref{t:K3K4} leads to the following.

\begin{cor} 
There exists an absolute constant $C$ such that every $(K_3, K_4)$-co-critical graph on $n\ge 9$ vertices has at least $7n-C$ edges. 
\end{cor}

  \section{Proof of \cref{t:main}}\label{s:main}
For a graph $G$ and a vertex $v\in V(G)$, let $N_G(v)$  be the set of vertices in $G$ that are adjacent to $v$.
For  $A\subseteq V(G)$,  the subgraph of $G$ induced by $A$, denoted $G[A]$, is the graph with vertex set $A$ and edge set $\{xy \in E(G)\mid x, y \in A\}$.  For another set $B\subseteq V(G)$  that is disjoint from $A$, we denote by $B \less A$ the set $B - A$, $e_G(A, B)$ the number of edges between $A$ and $B$ in $G$, $e_G(A)$ the number of edges of $G[A]$, and $G \less A$ the subgraph of $G$ induced on $V(G) \less A$, respectively.\medskip

Let   $G$ be a  $(K_{t_1},\ldots,K_{t_k})$-co-critical graph and  let  $\tau : E(G) \rightarrow [k]$ be a critical $k$-coloring of $G$.  Then $G$ contains no monochromatic copy of $K_{t_i}$ in color $i$ for any $i\in[k]$ under $\tau$. Given an nonempty set $R\subseteq V(G)$, to apply the $q$-neighbor bootstrap percolation closure method to $G$, we need to partition certain subset of  $V(G)\setminus R$ according to both their adjacencies to the vertices of $R$ and the colors assigned by $\tau$. More precisely, for each nonempty set $R\subseteq V(G)$  and each vertex $v\in V(G)\setminus R$, define a function $\tau_{R, v}: R\longrightarrow\{0,1,\ldots,k\}$ given by,  for each $x\in R$,  
\[\tau_{R, v}(x)=\begin{cases}
0, & xv\notin E(G),\\
i, & \tau(xv)=i.
\end{cases}
\]
Two vertices $u,v\in V(G)\less R$ have the same \emph{colored $R$-type under $\tau$} if $\tau_{R, u}(x)=\tau_{R, v}(x)$ for every $x\in R$. There are at most
 $(k+1)^{|R|}$ possible colored $R$-types.  The next lemma is the essential in applying the q-neighbor bootstrap percolation method.   
 
  \begin{lem}\label{l:cnbr} 
  Let $R\subseteq V(G)$ and let $x,y\in V(G)\setminus R$ have the same
colored $R$-type under $\tau$.  Let $z\in V(G)\setminus R$.  If  $zx\in E(G)$ and  $zy\notin E(G)$, then $z$ and $y$ have a common
neighbor in  $V(G)\setminus R$.
 \end{lem}
\begin{proof} We may assume that $\tau(zx)=i$ for some $i\in[k]$. We may further assume that the color $i$ is blue. Since $G$ is  $(K_{t_1},\ldots,K_{t_k})$-co-critical, we see that $G+zy$ admits no critical $k$-coloring, in particular, $G+zy$ contains a blue copy of $K:=K_{t_i}$ under $\tau'$, where $\tau'$ is obtained from $\tau$ by coloring the edge $zy$ blue. Then $y,z\in V(K)$. Since $x$ and $y$ have the same  colored $R$-type under $\tau$, we see that $xv$ is colored blue under $\tau$ for every  vertex  $v\in V(K)\cap R$. Recall that $xz$ is colored blue under $\tau$. Since $G$ contains no blue copy of $K$ under $\tau$, it follows that $|V(K)\less  R|\le t_i-3$ and so $V(K)\less  R\ne\emptyset$. Now every vertex in $V(K)\less  R$ is adjacent to both  $z$ and $y$ in $G$, as desired. 
\end{proof} 
 
 We are now ready to prove \cref{t:main}, which we restate here for convenience.\threeclaw*
\begin{proof}
Let $G$ be a  $(K_{t_1},\ldots,K_{t_k})$-co-critical graph on $n\ge r$ vertices. Assume that   $\delta(G)\ge r-2$. We next show that there exists a constant $C=C(r,k)$ such that $e(G)\ge (r-2)n-C$. We may assume that $r\ge3$. \medskip

 Let $\tau : E(G) \rightarrow [k]$ be  a critical $k$-coloring of $G$. Then $G$ contains no monochromatic copy of $K_{t_i}$ in color $i$ for any $i\in[k]$ under $\tau$.   It is easy to see that  $G$ is connected. Let   $q\in\mathbb N$ with $q:=r-2$. Then $\delta(G) \ge q$. Following  the proof of Theorem 7(h) in~\cite{SZ21}, we next apply the   $q$-neighbor     bootstrap percolation   on $G$. Given a set  $S \subseteq V(G)$ and  any vertex  $v\in V(G)$, let  $N_S(v): =N_G(v) \cap S$ and $d_S(v): =|N_S(v)|$. 
Let  $R \subseteq V(G)$ be  any nonempty set.   Let $R^0: =R$ and for   $i \ge 1$, let 
\[
R^i: =R^{i-1} \cup \{v \in V(G)\mid  d_{R^{i-1}}(v) \ge q\}.
\] 
 Let $\overline{R}: = \bigcup_{i \ge 0} R^i$, the closure of $R$ under the $q$-neighbor bootstrap percolation on $G$.   Then 
 \[e(G[\overline{R}]) \ge q(|\overline{R}|-|R|), \] 
 because every vertex in $R^i\less R^{i-1}$ is 
 adjacent to at least $q$ vertices in $R^{i-1}$.  
 Let  $Y(R): =V(G) \less \overline{R}$. Then $d_{\overline{R}}(v)\le q-1$ for each $v\in Y(R)$. Finally, for any $v \in V(G)$,  let 
  \[
  \omega_{_R}(v): =d_{\overline{R}}(v) + \frac{d_{Y(R)}(v)}2.
  \]
  We call $\omega_{_R}(v)$ the  weight of $v$ (with respect to $R$). Then 
  \[e_G(\overline{R}, Y(R)) +e_G(Y(R))= \sum_{v \in Y(R)} \omega_{_R}(v).\]
 Assume that there exists a constant $C_1(q, k)$ and  a nonempty set $R\subseteq V(G)$ with $|R|\le C_1(q, k)$ such that $\omega_{_R}(v)\ge q$ for all $v \in Y(R)$. Then  
  \[e_G(\overline{R}, Y(R)) +e_G(Y(R))= \sum_{v \in Y(R)} \omega_{_R}(v)\ge q|Y(R)|.\]
Therefore, 
\begin{align*}
  e(G) & = e(G[\overline{R}])+e_G(\overline{R}, Y(R)) +e_G(Y(R))\\
  & \ge q(|\overline{R}|-|R|)+q|Y(R)|\\
& \ge q(|\overline{R}|+|Y(R)|)-q|R|)\\
& \le qn-qC_1(q,k)\\
&=qn-C
\end{align*}
where $C=qC_1(q,k)$, as desired.  
\medskip

Within $Y(R)$, we define $B(R):=\{v \in Y(R)\mid  \omega_{_R}(v) < q\}$, which we call the set of bad vertices (with respect to $R$). It suffices to show that   there exists a constant $C_1(q, k)$  and a nonempty set $R\subseteq V(H)$ with $|R|\le C_1(q, k)$ such that $B(R)=\es$.   \medskip 
   
  Assume $B(R)\ne \es$ for our initial $R$. Our goal is  to  move a small number of vertices into $R$ so that the remaining vertices in
$B(R)$ have strictly larger weight.  To achieve this,  we first observe  that for each $v\in B(R)$, $d_G(v)\ge q$ and $ \omega_{_R}(v)=d_{\overline{R}}(v) + \frac{d_{Y(R)}(v)}2<q$. It follows that \medskip

\noindent ($\ast$) \,\, $d_{Y(R)}(v)\ge1$ for every vertex $v\in B(R)$. \medskip

\noindent Recall that there are at most $(k+1)^{|R|}$ possible colored $R$-types under $\tau$.   We now partition $B(R)$ according to the colored $R$-types under $\tau$.  Let $p$ denote the number of non-empty colored $R$-type classes of $B(R)$. Then $p\le (k+1)^{|R|}$. 
Choose one vertex $y_j$ from each of the $p$ classes. By ($\ast$), let  $x_j\in Y(R)$ such that $x_jy_j\in E(G)$.   Let $X(R):=\{x_1, \ldots, x_p\}$. 
Finally, let 
  \[N(R):=\{v \in \overline{R}\mid  vx\in E(G)  \ \text{ for some } x\in X(R)\}.\]
Then $|N(R)|\le (q-1)|X(R)|$ because $d_{\overline{R}}(v)\le q-1$ for each $v\in Y(R)$. 
   We next show that  {\bf Algorithm}~\ref{algo} below yields   a nonempty set   $R\subseteq V(G)$ with $B(R)=\emptyset$.\bigskip

\begin{algorithm}[H]\label{algo}
\SetAlgoLined
\KwData{  $G$   with  $\delta(G) \ge q$  }
\KwResult{A nonempty set $R\subseteq V(G)$ with $B(R)=\emptyset$}
Set  $R $  to be a set containing an arbitrary  vertex in $G$;

\While{$B(R)\neq \emptyset$ }{
Set $R$ to be $R \cup X(R) \cup N(R)  $\;
}

\caption{Building a nonempty set $R\subseteq V(G)$ with $B(R)=\emptyset$}
\end{algorithm} 
\bigskip

\noindent Let $R_i$ be the set $R$ obtained  in the $i$-th iteration  of {\bf Line 2}  when  running {\bf Algorithm}~\ref{algo}. Then for all $i\ge1$, $R_{i-1}\subseteq R_i$,    $\overline{R}_{i-1} \subseteq \overline{R}_i$, $Y(R_i)\subseteq Y(R_{i-1})$ and $B(R_i) \subseteq B(R_{i-1})$. To see why $\omega_{_{R_i}}(v)\ge \omega_{_{R_{i-1}}}(v)$ for all $v\in  B(R_i)$, we next introduce a control function on $V(G)$, because   dealing with $\omega_{_R}(v)$ directly is difficult.   Let   $\phi_{_R}(v): =\sum_{x \in N_G(v)} f_{_R}(x)$ for all $v\in V(G)$, where for all $x\in V(G)$, 
\[f_{_R}(x)=
\begin{cases}
1, & \text{if}\ x \in R, \\
1/2, & \text{if}\ x \in \overline{R} \less R,\\
d_{R}(x)/(2q), & \text{if}\ x \in Y(R). 
\end{cases}
\]

\noindent It is worth noting that  $\phi_{_R}(v) \le \omega_{_R}(v)$ for every vertex  $v \in V(G)$, because    $d_{\overline{R}}(x) \le q-1$ for all  $x \in Y(R)$. Similarly,  for all $i\ge1$, $f_{_{R_{i-1}}}(x) \le f_{_{R_i}}(x)$ for every $x \in V(G)$,  because $Y(R_i)\subseteq Y(R_{i-1})$.  We next show that \bigskip

\noindent ($\ast\ast$)  for all $i \ge 1$ and  every $v \in B(R_i)$,  $\phi_{_{R_i}}(v) \ge \phi_{_{R_{i-1}}}(v)+1/(2q)$.\medskip

 Let $i\ge1$ and $v \in B(R_i)$. Then $v\in B(R_{i-1})$  since $B(R_i) \subseteq B(R_{i-1})$. Let  $X(R_{i-1})$ and $N(R_{i-1})$  be defined accordingly for $R_{i-1}$.   To prove $\phi_{_{R_i}}(v) \ge \phi_{_{R_{i-1}}}(v)+1/(2q)$, it suffices to show that $f_{_{R_i}}(x) \ge f_{_{R_{i-1}}}(x)+1/(2q)$ for some $x \in N_G(v)$. Note that $X(R_{i-1})\subseteq Y(R_{i-1})\cap R_i$. For each vertex  $x\in X(R_{i-1})$,  we see that $f_{_{R_{i-1}}}(x )=d_{_{R_{i-1}}}(x )/(2q) \le (q-1)/(2q)=1/2-1/(2q)$,  and $f_{_{R_i}}(x )=1>f_{_{R_{i-1}}}(x )+1/(2q)$.  We may assume that $vx \not\in E(G)$ for all $x\in X(R_{i-1})$, otherwise we are done.  Since $v\in B(R_i)$, by the choice of $X(R_{i-1})$ and  $N(R_{i-1})$,  there exists $y\in Y(R_{i-1})$ such that $v$ and $y$ have the same colored $R_{i-1}$-type under $\tau$, that is, $\tau_{R_{i-1}, v}(z)=\tau_{R_{i-1}, y}(z)$ for all $z\in B(R_{i-1})$.  By the choice of $X(R_{i-1})$, let $x\in X(R_{i-1})$ such that $xy\in E(G)$. Note that $vx\notin E(G)$ and $v\ne x$ because $x\in R_i$ and $v \in B(R_i)$. By Lemma~\ref{l:cnbr} applied to $R_{i-1}, v, y, x$,  there exists $w\in V(G)\less R_{i-1}$ such that $w$ is adjacent to both $v$ and $x$ in $G$.   Assume  that  $w \in \overline{R}_{i-1} \less R_{i-1}$. Then $f_{_{R_{i-1}}}(w)=1/2$ and $f_{_{R_i}}(w)=1$, and so $f_{_{R_i}}(w) \ge f_{_{R_{i-1}}}(w)+1/(2q)$, as desired.   We may assume that $w \in Y(R_{i-1})$. Then 
   \[f_{_{R_{i-1}}}(w)=d_{_{R_{i-1}}}(w )/(2q) \le (q-1)/(2q)=1/2-1/(2q).\] 
   Note that either $w \in \overline{R}_i$ or $w \in Y(R_i)$. 
Assume  $w \in \overline{R}_i$. Then $f_{_{R_i}}(w) \ge 1/2$ and so $f_{_{R_i}}(w) \ge f_{_{R_{i-1}}}(w)+1/(2q)$, as desired.   Finally, assume $w \in Y(R_i)$. Note that $wx\in E(G)$,  $x\in R_i\less R_{i-1}$ and $R_{i-1}\subseteq R_i$. It follows that $d_{R_i}(w)\ge d_{R_{i-1}}(w)+1$. Hence,  $f_{_{R_i}}(w)=d_{_{R_i}}(w)/(2q) \ge (d_{_{R_{i-1}}}(w)+1)/(2q)= f_{_{R_{i-1}}}(w)+1/(2q)$.    In all cases, we have shown that there exists  some  vertex $x\in N_G(v)$ such that     $f_{_{R_i}}(x)\ge f_{_{R_{i-1}}}(x)+1/(2q)$.  Therefore,  $\phi_{_{R_i}}(v) \ge \phi_{_{R_{i-1}}}(v)+1/(2q)$ for all $i\ge1$ and $v\in B(R_i)$. This proves  ($\ast\ast$).  \medskip
  
  By ($\ast\ast$),    {\bf Algorithm}~\ref{algo} stops  after   $m\le 2q^2$ iterations of {\bf Line 2} because $\phi_{_R}(v) \le \omega_{_R}(v)<q$ for each $v\in B(R)$. Hence  $R_m\subseteq V(G)$ with $  R_m\ne\es$ but  $B(R_m)=\emptyset$. For all $i \ge 0$, 
\begin{align*}
|R_{i+1}| &\le |R_i|+|X(R_i)| +|S(R_i)| \\
& \le |R_i|+|X(R_i)| +(q-1)|X(R_i)| \\
&\le |R_i|+ q(k+1)^{|R_i|}, 
\end{align*}
 which   only depends on $q$ and $k$. It follows that by {\bf Algorithm}~\ref{algo},  there exists  a constant $C_1(q,k)$  and a non-empty set $R\subseteq V(G)$ with   $|R| \le C_1(q,k)$  such that  $B(R)=\emptyset$, as desired.   \medskip
 
 This completes the proof of \cref{t:main}.
\end{proof}
 
\section{Concluding remarks}

 Galluccio,  Simonovits and Simonyi~\cite{Galluccio1992} also made an observation on the chromatic number of  $(K_{t_1}, \ldots,  K_{t_k})$-co-critical graphs.

\begin{lem}[Galluccio,  Simonovits and Simonyi~\cite{Galluccio1992}]\label{l:chi}
If  $G$ is a   $(K_{t_1}, \ldots,  K_{t_k})$-co-critical graph, then    $\chi(G)\ge r-1$,  and the equality holds  only when  $G$ is a complete $(r-1)$-partite graph, where $\chi(G)$ denotes the chromatic number of $G$ and $r=R(K_{t_1}, \ldots,  K_{t_k})$.  

\end{lem}

\cref{l:chi} is essential in the proof of \cref{t:K3K4}, and it  implies that if $G$ is a   $(K_{t_1}, \ldots,  K_{t_k})$-co-critical graph with $\delta(G)\le r-3$, then $\chi(G)\ge r$. \medskip
 
 Casas-Rocha, Snyder and the present author~\cite{CSS24} also established the following lower bound on the minimum degree of $(K_{t_1},\ldots,K_{t_k})$-co-critical graphs. However, this bound remains far from the bound $r-2$ given in \cref{c:mindeg}.

\begin{thm}[Casas-Rocha, Snyder and Song~\cite{CSS24}]
 For all integers $k\ge2$  and  $t_k\ge \cdots\ge t_1\ge 3$, every $(K_{t_1}, \ldots,K_{t_k})$-co-critical graph on $n\ge r$ vertices has minimum degree at least $t_k-2k-1+\sum_{i=1}^k t_i$. 
\end{thm}

\end{document}